\RequirePackage{etoolbox}
\csdef{input@path}{{style/}{graphics/}}
\documentclass[MSNbibl,number,citesort,dvips]{arxbj}
\usepackage{upgreek}
\usepackage{graphicx}

\aid{0}
\volume{21}
\issue{1}
\pubyear{2015}
\firstpage{38}
\lastpage{61}
\doi{10.3150/13-BEJ560} 

\makeatletter
\newcommand{\rrvert}{\vert}
\newcommand{\llvert}{\vert}
\newtheorem{prop}{Proposition}[section]
\newtheorem{cor}[prop]{Corollary}
\newtheorem{theorem}[prop]{Theorem}
\newremark{ex}{Example}
\newremark{rem}[prop]{Remark}
\newproclaim{deff}[prop]{Definition}

\newcommand{\eqref}[1]{(\ref{#1})}

\newcommand{\diag}{\operatorname{diag}}
\newcommand{\R}{\mathbb R}
\newcommand{\E}{\mathbb E}
\newcommand{\N}{\mathbb N}
\renewcommand{\P}{\mathbb P}

\def\sfrac#1#2{#1/#2}
\def\vfrac#1#2{(#1)/#2}
\def\afrac#1#2{#1/(#2)}

\def\sklfrac#1#2{(#1/#2)}
\renewcommand{\epsilon}{\varepsilon}
\makeatother

\begin{document}
\begin{frontmatter}

\title{Maxima of independent, non-identically distributed Gaussian vectors}
\runtitle{Maxima of independent Gaussian vectors}

\begin{aug}
\author[1,2]{\inits{S.}\fnms{Sebastian} \snm{Engelke}\corref{}\thanksref{1,2}\ead[label=e1]{sengelk@uni-goettingen.de}},
\author[3]{\inits{Z.}\fnms{Zakhar} \snm{Kabluchko}\thanksref{3}\ead[label=e2]{zakhar.kabluchko@uni-ulm.de}} \and
\author[4]{\inits{M.}\fnms{Martin}~\snm{Schlather}\thanksref{4}\ead[label=e3]{schlather@math.uni-mannheim.de}}
\address[1]{Institut f\"ur Mathematische Stochastik,
Georg-August-Universit\"at G\"ottingen, Goldschmidtstr. 7, D-37077 G\"
ottingen, Germany.
\printead{e1}}
\address[2]{Facult\'e des Hautes Etudes Commerciales, Universit\'e de
Lausanne, Extranef, UNIL-Dorigny, CH-1015 Lausanne, Switzerland}
\address[3]{Institut f\"ur Stochastik, Universit\"at Ulm,
Helmholtzstr. 18, D-89069 Ulm, Germany.\\
\printead{e2}}
\address[4]{Institut f\"ur Mathematik, Universit\"at Mannheim, A5, 6,
D-68131 Mannheim, Germany.\\
\printead{e3}}
\end{aug}

\received{\smonth{5} \syear{2012}}
\revised{\smonth{8} \syear{2013}}

%
\begin{abstract}
Let $X_{i,n}, n\in\N, 1\leq i \leq n$, be a triangular array of
independent $\R^d$-valued Gaussian
random vectors with correlation matrices $\Sigma_{i,n}$.
We give necessary conditions under which the row-wise maxima converge
to some max-stable distribution
which generalizes the class of H\"usler--Reiss distributions. In the
bivariate case, the conditions
will also be sufficient. Using these results, new models for bivariate
extremes are derived explicitly.
Moreover, we define a new class of stationary, max-stable processes as
max-mixtures of Brown--Resnick
processes. As an application, we show that these processes realize a
large set of extremal correlation
functions, a natural dependence measure for max-stable processes. This
set includes all functions
$\psi(\sqrt{\gamma(h)}),h\in\R^d$, where $\psi$ is a completely
monotone function and $\gamma$ is an arbitrary variogram.
\end{abstract}

%
\pdfkeywords{extremal correlation function, Gaussian random vectors, Husler-Reiss distributions, max-limit theorems, max-stable distributions, triangular arrays}
\end{frontmatter}

\section{Introduction}

It is well known that the standard normal distribution $\Phi$ is in
the max-domain of attraction of the Gumbel distribution, that is,
\begin{eqnarray*}
\lim_{n\to\infty} \Phi(b_n + x / b_n)^n
= \exp\bigl(-\exp(-x)\bigr), \qquad \mbox {for all } x\in\R,
\end{eqnarray*}
where $b_n, n\in\N$, is a sequence of normalizing constants defined
by $b_n=n\phi(b_n)$, where $\phi$ is the standard normal density.
By Theorem~1.5.3 in Leadbetter \textit{et al.} \cite{lea1983}, it is given as
%
\begin{eqnarray}
\label{bn} b_n := \sqrt{2\log n} - \frac{(1/2)\log\log n + \log(2\sqrt{\uppi
})}{\sqrt{2\log n}} + \mathrm{o} \bigl((
\log n)^{-1/2} \bigr).
\end{eqnarray}
Sibuya \cite{sib1960} showed that the maxima of i.i.d.
bivariate normal random vectors with correlation $\rho<1$
asymptotically always become independent.
However, for triangular arrays with i.i.d. entries within each row
where the correlation in the $n$th row approaches
$1$, as $n\to\infty$, with an appropriate speed, H{\"u}sler and Reiss \cite{hue1989} proved that the row-wise maxima converge to a
new class of max-stable bivariate distributions, namely
%
\begin{eqnarray}
\label{hr_distr} F_\lambda(x,y) = \exp \biggl[ -\Phi \biggl(\lambda+
\frac
{x-y}{2\lambda} \biggr)\mathrm{e}^{-y} -\Phi \biggl(\lambda+
\frac
{y-x}{2\lambda} \biggr)\mathrm{e}^{-x} \biggr], \qquad x,y\in\R.
\end{eqnarray}
Here, $\lambda\in[0,\infty]$ parameterizes the dependence in the
limit, $0$ and $\infty$ corresponding to complete dependence and
asymptotic independence, respectively. In fact, Kabluchko \textit{et al.}
\cite{kab2009} provide a simple argument that these are also the only
possible limit points for such triangular arrays.

More generally, H{\"u}sler and Reiss \cite{hue1989} consider
triangular arrays with i.i.d. entries of $d$-variate zero-mean,
unit-variance normal random vectors with correlation matrix $\Sigma_n$
in the $n$th row satisfying
%
\begin{eqnarray}
\label{sigma_cond} \lim_{n\to\infty} \log n \bigl(\mathbf{11}^\top-
\Sigma_n\bigr) = \Lambda \in[0,\infty)^{d\times d},
\end{eqnarray}
where $\mathbf{1} = (1,\dots,1)^\top\in\R^d$ and$\phantom{x}^\top
$ denotes the transpose sign.
Under this assumption, the row-wise maxima converge to the $d$-variate,
max-stable H\"usler--Reiss
distribution whose dependence structure is fully characterized by the
matrix $\Lambda$.
Note that condition \eqref{sigma_cond} implies that all off-diagonal
entries of $\Sigma_n$ converge to $1$ as $n\to\infty$.
A slightly more general representation is given in Kabluchko \cite{kab2011} in terms of Poisson point processes and
negative definite kernels.

In fact, it turns out that these distributions not only attract
Gaussian arrays but also classes of related distributions.
For instance, Hashorva \cite{has2005} shows, that the
convergence of maxima holds for triangular arrays of general bivariate
elliptical distributions, if the random radius is in the domain of
attraction of the Gumbel distribution.
The generalization to multivariate elliptical distributions can be
found in Hashorva \cite{has2006}. Moreover,
Hashorva \textit{et al.} \cite{has2012a} prove that also some
non-elliptical distributions are in the domain of attraction of
the H\"usler--Reiss distribution, for instance multivariate $\chi
^2$-distributions.

Apart from being one of the few known parametric families of
multivariate extreme value distributions,
the H\"usler--Reiss distributions play a prominent role in modeling
spatial extremes since they are the
finite-dimensional distributions of Brown--Resnick processes  \cite{bro1977,kab2009}.

Recently, Hashorva and Weng \cite{has2013} analyzed maxima of
stationary Gaussian triangular arrays where the variables in each row
are identically distributed but not necessarily independent. They show
that weak dependence is asymptotically negligible, whereas stronger
dependence may influence the max-limit distribution.

In this paper, we consider independent triangular arrays $\mathbf
{X}_{i,n} = (X^{(1)}_{i,n},\dots,X^{(d)}_{i,n})$, $n\in\N$ and
$1\leq i\leq n$, where $\mathbf{X}_{i,n}$ is a zero-mean,
unit-variance Gaussian random vector with correlation matrix $\Sigma
_{i,n}$. Thus, in each row the random variables are independent, but
may have different dependence structures. Letting $\mathbf{M}_n =
(M_n^{(1)},\dots,M_n^{(d)})$ denote the vector consisting of the
componentwise maxima
$M_n^{(j)} = \max_{i=1,\dots,n} X^{(j)}_{i,n}$, $j\in\{1,\dots,d\}
$, we are interested in the convergence of the rescaled, row-wise maximum
%
\begin{eqnarray}
\label{row_max} b_n(\mathbf{M}_n - b_n),
\end{eqnarray}
as $n\to\infty$, and the respective limit distributions.

In Section~\ref{sec_bivariate}, we start with bivariate triangular
arrays. For this purpose, we introduce a sequence of counting measures
which capture the dependence structure in each row and which is used to
state necessary and sufficient conditions for the convergence of \eqref
{row_max}. Moreover, the limits turn out to be new max-stable
distributions that generalize \eqref{hr_distr}. The results on
triangular arrays are used to completely characterize the max-limits of
independent, but not necessarily identically distributed sequences of
bivariate Gaussian vectors. Explicit examples for the bivariate limit
distributions are given at the end of Section~\ref{sec_bivariate}. The
multivariate case is treated in Section~\ref{sec_multivariate}, giving
rise to a class of $d$-dimensional max-stable distributions. In
Section~\ref{sec_processes}, we show how these distributions arise as
finite-dimensional margins of the new class of max-mixtures of
Brown--Resnick processes. Furthermore, it is shown that these processes
offer a large variety of extremal correlation functions which makes
them interesting for modeling dependencies in spatial extremes.
Finally, Section~\ref{sec_proofs} comprises the proofs of the main theorems.

\section{The bivariate case}
\label{sec_bivariate}

Before we start with bivariate triangular arrays, let us note that even
the case of univariate sequences of
independent yet non-identically distributed Gaussian random variables
is not trivial.
In fact, many different distributions for the max-limits may arise,
which are not necessarily max-stable (see Example~\ref{ex_univariate} below).
In the sequel, we will therefore restrict to the case that the
variances of the univariate margins
are close to some constant, which can be assumed to be $1$ without loss
of generality,
and we fix the normalization in \eqref{row_max}. Later, for the sake
of simplicity, we
will always consider margins with unit variance.

In order to state the main results in the bivariate case, we need
probability measures on the extended positive half-line $[0,\infty]$.
Endowed with the metric $d(x,y)=|\mathrm{e}^{-x}-\mathrm{e}^{-y}|$, the space $[0,\infty
]$ becomes compact. A function $g\dvtx [0,\infty]\to\R$ is continuous iff
it is continuous in the usual topology on $[0,\infty)$ and the limit
$\lim_{x\to\infty} g(x)$ exists and equals $g(\infty)$.

\subsection{Limit theorems}

Consider a triangular array of independent bivariate Gaussian random
vectors $\mathbf{X}_{i,n} = (X^{(1)}_{i,n},\allowbreak X^{(2)}_{i,n})$, $n\in\N$
and $1\leq i\leq n$, with zero expectation and covariance matrix
\begin{eqnarray*}
\operatorname{Cov} (\mathbf{X}_{i,n}) = \pmatrix{ \sigma^2_{i,n,1}
& \sigma_{i,n,1,2}
\cr
\sigma_{i,n,1,2} & \sigma^2_{i,n,2}
},
\end{eqnarray*}
with $\sigma^2_{i,n,j} > 0$ for all $n\in\N$, $1\leq i\leq n$ and
$j\in\{1,2\}$. Further, denote by $\rho_{i,n} = \sigma_{i,n,1,2} /\allowbreak
(\sigma_{i,n,1}\sigma_{i,n,2})$ the correlation of $\mathbf
{X}_{i,n}$. For $n\in\N$, we define a probability measure $\eta_n$
on $[0,\infty]\times\mathbb{R}^2$ by
%
\begin{eqnarray}
\label{def_eta_n} \eta_n = \frac{1}n \sum
_{i=1}^n \delta_{ (\sqrt{b_n^2(1-\rho
_{i,n})/2}, b_n^2(1-1/\sigma_{i,n,1}), b_n^2(1-1/\sigma
_{i,n,2}) )}
\end{eqnarray}
which encodes the suitably normalized variances and correlations in the
$n$th row. More precisely, it
maps the rate with which the variances and correlations converge to
$1$. Here, for any measurable space $(S,\mathcal{S})$ and $a\in S$,
$\delta_a$ denotes the Dirac measure on the point $a$.

In this general situation, the next theorem gives a sufficient
condition in terms of $\eta_n$ for the convergence of row-wise maxima
of this triangular array.

\begin{theorem}
\label{thm1}
For $n\in\N$ and $1\leq i\leq n$, let $\mathbf{X}_{i,n} $ and $\eta
_n$ be defined as above. Further suppose that for some $\epsilon>0$
the measures $(\eta_n)_{n\in\N}$ satisfy the integrability condition
%
\begin{eqnarray}
\label{unif_int} \sup_{n\in\N} \int_{[0,\infty]\times\R^2}
\bigl[\mathrm{e}^{\theta
(1+\epsilon)} + \mathrm{e}^{\gamma(1+\epsilon)} \bigr] \eta_n\bigl(\mathrm{d}(\lambda
,\theta,\gamma)\bigr) < \infty.
\end{eqnarray}
If for $n\to\infty$, $\eta_n$ converges weakly to some probability
measure $\eta$ on $[0,\infty]\times\R^2$, that is, $\eta_n
\Rightarrow\eta$, then
%
\begin{eqnarray}
\label{normed_max} \max_{i=1,\dots,n} b_n(\mathbf{X}_{i,n}
- b_n)
\end{eqnarray}
converges in distribution to a random vector with distribution function
$F_\eta$ given by
%
\begin{eqnarray}
\label{distr_gen}-\log F_\eta(x,y) &=& \int_{[0,\infty]\times\mathbb
{R}^2}
\Phi \biggl(\lambda+ \frac{y-x + \theta- \gamma}{2\lambda
} \biggr)\mathrm{e}^{-(x-\theta)}\nonumber
\\[-8pt]\\[-8pt]
&&\hphantom{\int_{[0,\infty]\times\mathbb
{R}^2}}{}
+ \Phi \biggl(\lambda- \frac{y-x + \theta- \gamma
}{2\lambda} \biggr)\mathrm{e}^{-(y-\gamma)}
\eta\bigl(\mathrm{d}(\lambda, \theta, \gamma)\bigr),\nonumber
\end{eqnarray}
for $x,y\in\R$.
\end{theorem}

\begin{rem}
Condition \eqref{unif_int} implies
\begin{eqnarray*}
\sup_{n\in\N, 1\leq i \leq n} \frac{1}n \bigl(\mathrm{e}^{b_n^2(1-1/\sigma
_{i,n,1})(1+\epsilon)} +
\mathrm{e}^{b_n^2(1-1/\sigma_{i,n,2})(1+\epsilon
)} \bigr) < \infty.
\end{eqnarray*}
Since $b_n^2 \sim2\log n$ for $n$ large, it follows that the variances
of both components
are uniformly bounded. Thus, the single random variables in each row
satisfy the uniform asymptotical negligibility condition (see, for
instance, \cite{bal1977})
%
\begin{eqnarray}
\label{asymp_neg} \max_{i=1,\dots, n} \P\bigl(b_n
\bigl(X_{i,n}^{(j)} - b_n\bigr) > x\bigr) \to0,\qquad  n\to
\infty,
\end{eqnarray}
for $j=1,2$ and any $x\in\R$.
\end{rem}

\begin{rem}
In fact, one can extend the distribution $F_\eta$ to mixture measures
$\eta$ taking infinite mass at negative infinity. The only condition
which needs to be satisfied is
\[
\int_{[0,\infty]\times\R^2} \bigl[\mathrm{e}^{\theta} + \mathrm{e}^{\gamma} \bigr]
\eta\bigl(\mathrm{d}(\lambda,\theta,\gamma)\bigr) < \infty.
\]
\end{rem}

\begin{rem}
Random variables with variances bounded away from $1$ from above do not
influence the maximum in the limit of \eqref{normed_max}. This can
easily be seen by allowing
weak convergence of $\eta_n$ to $\eta$ on the extended space
$[0,\infty]\times[-\infty,\infty)^2$.
\end{rem}

Note that the one-dimensional marginals of $F_\eta$ are Gumbel
distributed with certain location parameters, for instance,
\begin{eqnarray*}
-\log F_\eta(x,\infty) = \exp \biggl[ -x + \log\int
_{[0,\infty
]\times\mathbb{R}^2} \mathrm{e}^\theta \eta\bigl(\mathrm{d}(\lambda, \theta, \gamma)
\bigr) \biggr].
\end{eqnarray*}
Moreover, $F_\eta$ is a max-stable distribution since
\[
F_\eta^n(x + \log n, y + \log n) = F_\eta(x,y),
\]
for all $n\in\N$. This is a remarkable fact, since, in general,
limits of row-wise maxima of triangular arrays are not max-stable, not
even if the random variables in each row are identically distributed.

The idea of constructing extreme value distributions as in \eqref
{distr_gen} is not new. Indeed, it is well known
that any mixture of spectral measures is again a spectral measure. In
our case, however, these mixture distributions also
arise naturally as the max-limits of independent Gaussian triangular
arrays.

If we assume that the margins have variance $1$, that is, $\sigma
_{i,n,1}=\sigma_{i,n,2}=1$, we can obtain a necessary and sufficient
condition for the convergence of maxima. We denote by $\mathcal
{M}_1([0,\infty])$ the space of all probability measures on $[0,\infty
]$ endowed with the topology of weak convergence. By Helly's theorem,
this space is compact.

\begin{theorem}
\label{thm2}
Consider a triangular array of independent bivariate Gaussian random
vectors $\mathbf{X}_{i,n} = (X^{(1)}_{i,n},X^{(2)}_{i,n})$, $n\in\N$
and $1\leq i\leq n$, where $X^{(1)}_{i,n}$ and $X^{(2)}_{i,n}$ are
standard normal random variables. Denote by $\rho_{i,n}$ the
correlation of $\mathbf{X}_{i,n}$. Let
%
\begin{eqnarray}
\label{emp_meas}\nu_n = \frac{1}n \sum
_{i=1}^n \delta_{\sqrt
{b_n^2(1-\rho_{i,n})/2}}
\end{eqnarray}
be a probability measure on $[0,\infty]$. For $n\to\infty$,
%
\begin{eqnarray}
\label{nec_cond} \max_{i=1,\dots,n} b_n(
\mathbf{X}_{i,n} - b_n)
\end{eqnarray}
converges in distribution if and only if $\nu_n$ converges weakly to
some probability measure $\nu$ on $[0,\infty]$, that is, $\nu_n
\Rightarrow\nu$. In this case, the limit of \eqref{nec_cond} has
distribution function $F_\nu$ given by
%
\begin{eqnarray}
\label{distribution2} -\log F_\nu(x,y) = \int_0^\infty
\biggl[\Phi \biggl(\lambda+ \frac
{y-x}{2\lambda} \biggr)\mathrm{e}^{-x} +
\Phi \biggl(\lambda+ \frac
{x-y}{2\lambda} \biggr)\mathrm{e}^{-y} \biggr] \nu(\mathrm{d}
\lambda) ,
\end{eqnarray}
$x,y\in\R$. The distribution in \eqref{distribution2} uniquely
determines the measure $\nu$, that is, for two probability measures
$\nu, \tilde{\nu} \in\mathcal{M}_1([0,\infty])$ with $\nu\neq
\tilde{\nu}$ it follows that $F_\nu\neq F_{\tilde{\nu}}$.
Furthermore, $F_\nu$ depends continuously on $\nu$, in the sense that
if $\nu_n \Rightarrow\nu$, as $n\to\infty$, and $\nu_n, \nu\in
\mathcal{M}_1([0,\infty])$, then $F_{\nu_n}$ converges pointwise to
$F_\nu$.
\end{theorem}

\begin{rem}
\label{rem2}
If $\nu$ is a probability measure on $[0,\infty)$, an alternative
construction of the distribution $F_\nu$ is the following \cite{kab2011}, Section~3: Let $\sum_{i=1}^\infty\delta
_{U_i}$ be a Poisson point process on $\R$ with intensity $\mathrm{e}^{-u}\,\mathrm{d}u$
and suppose that $B$ has the normal distribution $N(-2S^2,4S^2)$ with
random mean and variance, where $S$ is $\nu$-distributed. Then, for a
sequence $(B_i)_{i\in\N}$ of i.i.d. copies of $B$, the bivariate
random vector $\max_{i\in\N} (U_i, U_i + B_i)$ has distribution
$F_\nu$.
\end{rem}

\begin{ex}
\label{rem1}
For an arbitrary probability measure $\nu\in\mathcal{M}_1([0,\infty
])$, let $(R_i)_{i\in\N}$ be a sequence of i.i.d. samples of $\nu$.
Putting $\rho_{i,n} = \max(1 - 2R_i^2/b_n^2,-1)$ in Theorem~\ref
{thm2} yields
\begin{eqnarray*}
\nu_n = \frac{1}n \sum_{i=1}^n
\delta_{\min(R_i,b_n)} \Rightarrow \nu, \qquad \mbox{a.s.},
\end{eqnarray*}
by the law of large numbers. Hence, \eqref{nec_cond} converges a.s. in
distribution to $F_\nu$.
\end{ex}

The above theorem can be applied to completely characterize the
distribution of the maxima of a \emph{sequence}
of independent, but not necessarily identically distributed bivariate
Gaussian random vectors with unit variance.

\begin{cor}
\label{cor1}
Suppose that $\mathbf{X}_{i} = (X^{(1)}_{i},X^{(2)}_{i})$, $n\in\N$
and $1\leq i\leq n$, is a sequence of independent bivariate Gaussian
random vectors where $X^{(1)}_{i}$ and $X^{(2)}_{i}$ are standard
normal random variables. Denote by $\rho_{i}$ the correlation of
$\mathbf{X}_{i}$ and let
\begin{eqnarray*}
\nu_n = \frac{1}n \sum_{i=1}^n
\delta_{\sqrt{b_n^2(1-\rho_{i})/2}}
\end{eqnarray*}
be a probability measure on $[0,\infty]$. For $n\to\infty$,
%
\begin{eqnarray}
\label{seq_cond} \max_{i=1,\dots,n} b_n(
\mathbf{X}_{i} - b_n)
\end{eqnarray}
converges in distribution if and only if $\nu_n$ converges weakly to
some probability measure $\nu$ on $[0,\infty]$. In this case, the
limit of \eqref{seq_cond} has distribution function $F_\nu$ as in
\eqref{distribution2}. Furthermore, for all $\nu\in\mathcal
{M}_1([0,\infty])$, $F_{\nu}$ is attained as a limit of \eqref
{seq_cond} for a suitable sequence $(\mathbf{X}_{i})_{i\in\N}$.
\end{cor}

\begin{rem}
It is worthwhile to note that, in general, the class of
max-selfdecomposable distributions
(cf. Mejzler \cite{mej1956}, de Haan and Ferreira \cite{deh2006a}, Theorem~5.6.1), that is, the max-limits of
\emph{sequences} of independent (not necessarily identically
distributed) random variables,
is a proper subclass of max-infinitely-divisible distributions, that
is, the max-limits of
\emph{triangular arrays} with i.i.d. random variables in each row. The
latter coincides with the
class of max-limits of \emph{triangular arrays}, where the rows are
merely independent but not
identically distributed  \cite{bal1977,ger1986}. In the (bivariate) Gaussian
case, the above shows that the max-limits of i.i.d. \emph{triangular
arrays}, namely the H\"usler--Reiss distributions in \eqref{hr_distr},
are a proper subclass of max-limits of independent \emph{triangular
arrays}, namely the distributions in \eqref{distribution2}, which, on
the other hand, coincide with the max-limits of independent \emph{sequences}.
\end{rem}

\begin{ex}\label{ex_univariate}
The following example shows that without any assumptions on the
variances, even the univariate case is not trivial.
Let $X_i$, $i\in\N$, be a sequence standard normal distribution.
Define the sequence of
maxima
\begin{eqnarray*}
M_n = \max_{i=1,\dots,n} X_i/i,\qquad  n\in\N.
\end{eqnarray*}
Clearly, $M_n$ converges almost surely to the non-degenerate
random variable $\max_{i\in\N} X_i/i$, which however is not an extreme
value distribution.
\end{ex}

\subsection{Examples}

In multivariate extreme value theory, it is important to have flexible
and tractable models for dependencies of extremal events.
The max-stable distributions $F_\nu$ in Theorem~\ref{thm2} for $\nu
\in\mathcal{M}_1([0,\infty])$ are max-mixtures of H\"usler--Reiss
distributions with different dependency parameters. They constitute a
large class of new bivariate max-stable distributions. We derive two of
them explicitly by evaluating the integral in \eqref{distribution2}.

\begin{ex}[(Rayleigh distributed $\nu$)]
The Rayleigh distribution has density
%
\begin{eqnarray}
\label{f_rayleigh} f_\sigma(\lambda) = \frac{\lambda}{\sigma^2} \mathrm{e}^{-\afrac{\lambda
^2}{2\sigma^2}},\qquad
\lambda\geq0,
\end{eqnarray}
for $\sigma> 0$. Choosing the dependence parameter $\lambda$
according to the Rayleigh distribution $\nu_\sigma$, we obtain the
bivariate distribution function
%
\begin{eqnarray}
\label{F_rayleigh} -\log F_{\nu_\sigma}(x,y) = \int_0^\infty
\biggl[\Phi \biggl(\lambda + \frac{y-x}{2\lambda} \biggr)\mathrm{e}^{-x} +
\Phi \biggl(\lambda+ \frac
{x-y}{2\lambda} \biggr)\mathrm{e}^{-y} \biggr]
\frac{\lambda}{\sigma^2} \mathrm{e}^{-\afrac{\lambda^2}{2\sigma^2}} \,\mathrm{d}\lambda,
\end{eqnarray}
for $x,y\in\R$. In order to evaluate this integral, we apply partial
integration
and use formulae 3.471.9 and 3.472.3 in Gradshteyn and Ryzhik \cite
{gra2000}. Equation \eqref{F_rayleigh} then simplifies to
%
\begin{eqnarray}
\label{rayleigh_mix}F_{\nu_\sigma}(x,y) = \exp \biggl[-\mathrm{e}^{-\min
(x,y)} -
\frac{1}{\eta} \mathrm{e}^{-\vfrac{y+x}{2}}\mathrm{e}^{-\sfrac{|y-x|\eta
}{2}} \biggr],\qquad  x,y\in\R,
\end{eqnarray}
where $\eta= \sqrt{1 + 1/\sigma^2}\in(1,\infty)$. Note that
$\sigma$ parameterizes the dependence of $F_{\nu_\sigma}$. As
$\sigma$ goes to $0$ (i.e., $\eta$ goes to $\infty$) the margins
become equal. On the other hand, as $\sigma$ goes to $\infty$ (i.e.,
$\eta$ goes to $1$) the margins become completely independent.
The corresponding Pickands' dependence function is given by
\begin{eqnarray*}
A_{\nu_\sigma}(\omega) &=& -\log F_{\nu_\sigma}\bigl(-\log\omega, -\log(1-
\omega)\bigr)
\\
&= &\max(\omega, 1 - \omega) + \frac{1}\eta\sqrt{\omega(1-\omega)} \max
\biggl\{\frac{\omega}{1-\omega},\frac{1-\omega}{\omega
} \biggr\}^{-\eta/2}, \qquad \omega
\in[0,1].
\end{eqnarray*}
\end{ex}

\begin{ex}[(Type-2 Gumbel distributed $\nu$)]
The Type-2 Gumbel distribution has density
\begin{eqnarray*}
f_b(\lambda) = 2b\lambda^{-3}\mathrm{e}^{-\sfrac{b}{\lambda^2}},\qquad
\lambda\geq0,
\end{eqnarray*}
for $b > 0$. With similar arguments as for the Rayleigh distribution
the distribution function $F_{\nu_b}$, where $\nu_b$ has density
$f_b$, is given by
\begin{eqnarray*}
F_{\nu_b}(x,y) = \exp \bigl[ -\mathrm{e}^{-x} -
\mathrm{e}^{-y} + \mathrm{e}^{-\vfrac
{y+x}{2}}\mathrm{e}^{-\sqrt{ (\vfrac{y-x}{2} )^2+ 2b}} \bigr],\qquad  x,y\in\R.
\end{eqnarray*}
In this case, the parameter $b\in(0,\infty)$ interpolates between
complete independence and complete dependence of the bivariate
distribution. In particular, if $b\to0$, then the margins are equal
and, on the other hand, if $b\to\infty$ then the margins are independent.
Here, Pickands' dependence function is
\begin{eqnarray*}
A_{\nu_b}(\omega) = 1 - \sqrt{\omega(1-\omega)} \exp \biggl(- \sqrt{
\biggl(\frac{\log(\omega/(1-\omega))}{2} \biggr)^2+ 2b} \biggr) , \qquad \omega\in[0,1].
\end{eqnarray*}
\end{ex}

Every multivariate max-stable distribution admits a spectral
representation \cite{res2008}, Chapter~5, where
the spectral measure contains all information about the extremal
dependence. Recently, Cooley \textit{et al.} \cite{coo2010} and
Ballani and Schlather \cite{bal2011} constructed new parametric models
for spectral measures. For the bivariate H\"usler--Reiss distribution,
de Haan and Pereira \cite{deh2006} give an explicit form of its
spectral density on the positive sphere $S^1_+=\{(x_1,x_2)\in[0,\infty
)^2, x_1^2 + x_2^2 =1\}$. More precisely, they show that for $\lambda
\in(0,\infty)$
\begin{eqnarray*}
-\log F_\lambda(x,y) = \int_0^{\uppi/2} \max
\bigl\{ \mathrm{e}^{-x}\sin \theta,\mathrm{e}^{-y}\cos\theta \bigr\}
s_\lambda(\theta) \,\mathrm{d}\theta,\qquad  x,y\in\R,
\end{eqnarray*}
and give a rather complicated expression for $s_\lambda$.
Using the equation
\begin{eqnarray*}
\phi \biggl(\lambda- \frac{\log\tan\theta}{2\lambda} \biggr) = \frac{\sin\theta}{\cos\theta}\phi \biggl(
\lambda+ \frac{\log
\tan\theta}{2\lambda} \biggr), \qquad \lambda\in(0,\infty), \theta\in [0,\uppi/2],
\end{eqnarray*}
their expression can be considerably simplified and the spectral
density becomes
\begin{eqnarray*}
s_\lambda(\theta) = \frac{1}{2\lambda\sin\theta\cos
^2\theta} \phi \biggl(\lambda+
\frac{\log(\tan\theta)}{2\lambda
} \biggr), \qquad \theta\in[0,\uppi/2].
\end{eqnarray*}
For the spectral density $s_\nu$ of the H\"usler--Reiss mixture
distribution $F_\nu$ as in \eqref{distribution2}, where $\nu$ does
neither have an atom at $0$ nor at $\infty$, we have the relation
\begin{eqnarray*}
s_\nu(\theta) = \int_0^\infty
s_\lambda(\theta) \nu(\mathrm{d}\lambda), \qquad \theta\in[0,\uppi/2].
\end{eqnarray*}

For the two examples above, we can compute the corresponding spectral densities.
\label{prop_spec}

\setcounter{ex}{2}
%
\begin{ex}[(continued)]
For the Rayleigh distribution, $s_{\nu_\sigma}$ is given by
\begin{eqnarray*}
s_{\nu_\sigma}(\theta) = \frac{\mathrm{e}^{-\sklfrac{1}{\sqrt2}|\log\tan
\theta|\sqrt{1+1/\sigma^2}}}{4\sqrt{\sigma^4+\sigma^2} (\sin
\theta\cos\theta )^{3/2}}, \qquad \theta\in[0,\uppi/2].
\end{eqnarray*}
\end{ex}

\begin{ex}[(continued)]
For the Type-2 Gumbel distribution with parameter $b > 0$, the spectral
density has the form
\begin{eqnarray*}
s_{\nu_b}(\theta) = \frac{\mathrm{e}^{-u_b(\theta)}}{4 (\sin\theta
\cos\theta )^{3/2}} \biggl(1 - \frac{(\log\tan\theta
)^2}{4u_b(\theta)^2}
\biggr) \biggl(1 + \frac{1}{u_b(\theta)} \biggr),\qquad  \theta\in[0,\uppi/2],
\end{eqnarray*}
with $u_b(\theta) = \sqrt{ (\log\tan\theta )^2/4 + 2b}$.
\end{ex}

\begin{figure}

\includegraphics{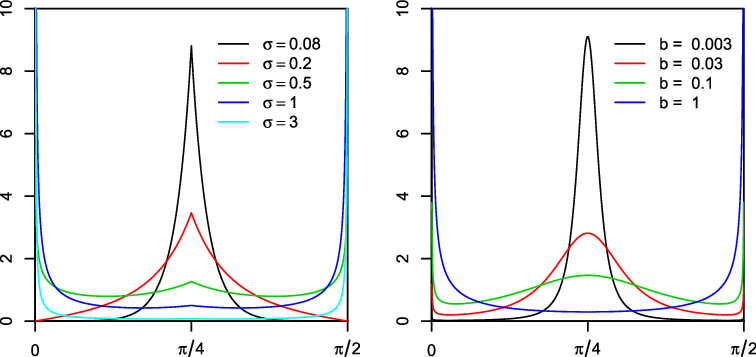}

\caption{Spectral densities of the Rayleigh (left) and Type-2 Gumbel
(right) mixture distribution for different parameters $\sigma$ and
$b$, respectively.}
\label{fig1}
\end{figure}

Figure~\ref{fig1} illustrates how these spectral measures interpolate
between complete independence and complete dependence for different parameters.

\section{The multivariate case}
\label{sec_multivariate}
Similarly as in H{\"u}sler and Reiss \cite{hue1989}, the results for
standard bivariate Gaussian random vectors can be generalized to
$d$-dimensional random vectors. To this end, define a triangular array
of independent $d$-dimensional
Gaussian random vectors $\mathbf{X}_{i,n} = (X^{(1)}_{i,n},\dots
,X^{(d)}_{i,n})$, $n,d\in\N$ and $1\leq i\leq n$,
where $X^{(j)}_{i,n}$, $j\in\{1,\dots,d\}$, are standard normal
random variables. Denote by
$\Sigma_{i,n} =  (\rho_{j,k}(i,n) )_{1\leq j,k\leq d}$ the
correlation matrix of $\mathbf{X}_{i,n}$.
Let $\mathbf{1} = (1,\ldots,1)^\top\in\R^d$ and
%
\begin{eqnarray}
\label{emp_meas_d}\eta_n = \frac{1}n \sum
_{i=1}^n \delta_{\sqrt
{b_n^2(\mathbf{11}^\top-\Sigma_{i,n})/2}}
\end{eqnarray}
be a probability measure on the metric space $[0,\infty)^{d\times d}$,
equipped with the Euclidean distance. Throughout this paper, squares
and square roots of matrices are to be understood component-wise. For a
measure $\tau$ on $[0,\infty)^{d\times d}$, we will denote by $\tau
^2$ the image measure of $\tau$ under the transformation $[0,\infty
)^{d\times d}\to[0,\infty)^{d\times d}$, $\Lambda\mapsto\Lambda
^2$. Further, let $D\subset[0,\infty)^{d\times d}$ be the subspace of
conditionally negative definite matrices which are symmetric and have
zeros on the diagonal, that is,
\begin{eqnarray*}
D &:= & \Biggl\{ (a_{j,k} )_{1\leq j,k\leq d} = A \in [0,
\infty)^{d\times d}\dvt  \mathbf{x}^\top A\mathbf{x} \leq0 \mbox{ for all
}\mathbf{x}\in\R^d\setminus\{0\} \mbox{ s.t. }
\\
&&\hphantom{\Biggl\{} \sum_{i=1}^d x_i = 0,
a_{j,k} = a_{k,j}, a_{j,j} = 0 \mbox{ for all } 1
\leq j,k\leq d \Biggr\},
\end{eqnarray*}
and let $D_0\subset D$ be the space of strictly conditionally negative
definite matrices, that is, where $\mathbf{x}^\top A\mathbf{x} < 0$
holds in
the above definition. In particular, note that $D_0$ is open in $D$ and
that $D$ is a suitable subspace for the measures $\eta^2_n$ since
$\eta_n^2(D) = 1$ for all $n\in\N$. For $\Lambda=  (\lambda
_{j,k} )_{1\leq j,k\leq d} \in[0,\infty)^{d\times d}$, define a
family of transformed matrices by
\[
\Gamma_{l,m} (\Lambda) = 2 \bigl(\lambda_{m_j,m_l}^2
+ \lambda _{m_k,m_l}^2 - \lambda_{m_j,m_k}^2
\bigr)_{1\leq j,k \leq l-1},
\]
where $2\leq l\leq d$ and $m = (m_1,\dots, m_l)$ with $1\leq m_1 <
\cdots< m_l\leq d$. It follows from the proof of Lemma~2.1 in
Berg \textit{et al.} \cite{ber1984} that if $\Lambda\in D_0$, then
$\Gamma_{l,m} (\sqrt{\Lambda})$ is a (strictly) positive definite
matrix.

Denote by $S( \cdot | \Psi)$ the so-called survivor function of an
$l$-dimensional normal random vector with mean vector $\mathbf{0}$ and
covariance matrix $\Psi$. That is, if $\mathbf{X}\sim N(\mathbf
{0},\Psi)$ and $\mathbf{x}\in\R^l$, then $S(\mathbf{x}|\Psi) = \P
( X_1 > x_1, \dots, X_l > x_l  )$. If $\Psi$ is not a
covariance matrix, we put $S(\mathbf{x}|\Psi) = 0$.

For a fixed $\Lambda=  ( \lambda_{j,k} )_{1\leq j,k \leq
d}\in[0,\infty)^{d\times d}$, let
\[
H_\Lambda(\mathbf{x}) = \exp \Biggl(\sum_{l=1}^d
(-1)^l \sum_{ m:
1\leq m_1 < \cdots< m_l\leq d } h_{l,m,\Lambda}(x_{m_1},
\dots, x_{m_l}) \Biggr),
\]
where
\[
h_{l,m,\Lambda}(y_1,\dots,y_l) = \int
_{y_l}^\infty S \bigl( \bigl( y_i - z + 2
\lambda^2_{m_i,m_l} \bigr)_{i=1,\dots,l-1} | \Gamma
_{l,m}(\Lambda) \bigr) \mathrm{e}^{-z} \,\mathrm{d}z,
\]
for $2\leq l\leq d$ and $h_{1,m,\Lambda}(y) = \mathrm{e}^{-y}$ for $m=1,\dots
,d$. For alternative representations of the multivariate H\"usler--Reiss
distribution $H_\Lambda$, see Joe \cite{joe1994}
and Kabluchko \cite{kab2011}. With this notation, we are
now in a position to state the following theorem.

\begin{theorem}
\label{multi_dim}
Consider a triangular array of independent $d$-dimensional Gaussian
random vectors as above.
If for $n\to\infty$ the measure $\eta_n$ in \eqref{emp_meas_d}
converges weakly to some probability measure $\eta$ on $[0,\infty
)^{d\times d}$, i.e., $\eta_n \Rightarrow\eta$, s.t. $\eta^2(D_0) =
1$, then
\begin{eqnarray*}
\max_{i=1,\dots,n} b_n(\mathbf{X}_{i,n} -
b_n)
\end{eqnarray*}
converges in distribution to a random vector with distribution function
$H_\eta$ given by
%
\begin{eqnarray}
\label{fdd} H_\eta(x_1,\dots,x_d) = \exp
\biggl(\int_{[0,\infty)^{d\times d}} \log H_\Lambda(x) \eta(\mathrm{d}\Lambda)
\biggr), \qquad x\in\R^d.
\end{eqnarray}
\end{theorem}

\begin{rem}
\label{rem_repr}
Similarly to Remark~\ref{rem2}, we can give an alternative
construction of the distribution $H_\eta$ in terms of Poisson point processes.
Let $\sum_{i=1}^\infty\delta_{U_i}$ be a Poisson point process on
$\R$ with intensity $\mathrm{e}^{-u}\,\mathrm{d}u$ and suppose that $\mathbf{B}$ has the
multivariate normal distribution $N(-\diag(\Gamma_{d,(1,\dots
,d)}(\Lambda))/2,\Gamma_{d,(1,\dots,d)}(\Lambda))$ with random mean
and variance, where $\Lambda$ is $\eta$-distributed. Then, for a
sequence $(\mathbf{B}_i)_{i\in\N}$ of i.i.d. copies of $\mathbf
{B}$, the random vector $\max_{i\in\N} (U_i, U_i + \mathbf{B}_i)$
has distribution $H_\eta$.
\end{rem}

\begin{rem}
We believe that the above theorem also holds in the case when $\eta$
has positive measure on \emph{non-strictly} conditionally negative
definite matrices, i.e., $\eta^2(D\setminus D_0) > 0$. Our proof of
this theorem however breaks down in this situation such that another
technique might be necessary.
\end{rem}

\begin{rem}
It is an open question if, similarly to the bivariate case, the
distribution $H_\eta$ uniquely determines the mixture measure $\eta$.
By Remark~\ref{rem_repr}, this problem is equivalent to the question
if the distribution of normal mixtures
$N(-\diag(\Gamma_{d,(1,\dots,d)}(\Lambda))/2,\break \Gamma_{d,(1,\dots
,d)}(\Lambda))$, where $\Lambda$ is $\eta$-distributed, determines
the measure $\eta$.
The solution of this problem is crucial to show that in Theorem~\ref
{multi_dim} the weak convergence $\eta_n\Rightarrow\eta$ is also
necessary for the
convergence of the maxima.
\end{rem}

\section{Application to Brown--Resnick processes}
\label{sec_processes}

The $d$-dimensional H\"usler--Reiss distributions arise in the theory of
maxima of Gaussian random
fields as the finite-dimensional distributions of the Brown--Resnick
process  \cite{bro1977} and its
generalizations  \cite{kab2009}. In this section,
we introduce a new class of max-stable processes
with finite-dimensional distributions given by \eqref{fdd} for
suitable measures $\eta$.

Let us briefly recall the definition of the processes introduced in
Kabluchko \textit{et al.} \cite{kab2009}. For a zero-mean
Gaussian process $\{W(t), t\in\R^d\}$ with stationary increments,
variance $\sigma^2(t)$ and
variogram $\gamma(t) = \E( W(t) - W(0))^2$, consider i.i.d. copies $\{
W_i, i\in\N\}$ of $W$ and a
Poisson point process $\sum_{i=1}^\infty\delta_{U_i}$ on $\R$ with
intensity $\mathrm{e}^{-u}\,\mathrm{d}u$, independent
of the family $W_i,i\in\N$. Kabluchko \textit{et al.} \cite{kab2009}
showed that the Brown--Resnick process
%
\begin{eqnarray}
\label{BRproc} \xi(t) = \max_{i\in\N} \bigl(U_i +
W_i(t) - \sigma(t)^2/2\bigr),\qquad  t\in\R^d,
\end{eqnarray}
is max-stable and stationary with standard Gumbel margins and that its
law depends only on the variogram $\gamma$.

We generalize this construction by allowing the variogram of the
Gaussian processes $W_i$ to be random. In fact, this defines
a new class of max-stable processes whose finite-dimensional
distributions are of the form \eqref{fdd}.

\begin{deff}
\label{def1}
Let $V_d$ be the measurable space of all variograms on $\R^d$, i.e.,
conditionally negative definite functions $\gamma$ on $\R^d$
with $\gamma(0) = 0$, equipped with the product $\sigma$-algebra.
Further, let $\mathbb{Q}$ be an arbitrary probability measure on this
space and $\gamma_i$, $i\in\N$, be an i.i.d. sequence of random
variables with distribution $\mathbb{Q}$.
For each $i\in\N$, let $W_i$ be a random field such that,
conditionally on $\gamma_i$, $W_i$ is
a zero-mean Gaussian process with stationary increments, variogram
$4\gamma_i$ and $W_i(0) = 0$ a.s. Further, let $\sum_{i=1}^\infty
\delta_{U_i}$ be a Poisson point process on $\R$ with intensity
$\mathrm{e}^{-u}\,\mathrm{d}u$ which is independent of the $W_i$, $i\in\N$.
Then, the process $\xi_{\mathbb{Q}}$ given by
\begin{eqnarray*}
\xi_{\mathbb{Q}}(t) = \max_{i\in\N} \bigl(U_i +
W_i(t) - 2\gamma _i(t)\bigr),\qquad  t\in\R^d,
\end{eqnarray*}
is called a max-mixture of Brown--Resnick processes w.r.t. the mixture
measure $\mathbb{Q}$.
\end{deff}

Note that a different kind of process can be defined
through a hierarchical or Bayesian approach, which is not
considered here and which does not lead to a max-stable
process, in general: first, exactly one realization of the
variogram is drawn from $\mathbb Q$. Then, conditionally
on this realization, a Brown--Resnick process is
simulated. Obviously, the resulting process must
lie in the max-domain of attraction of the process given
in Definition~\ref{def1}, with the same law $\mathbb Q$ for the variograms.
This implies immediately the following proposition;
a direct proof is given in Section \ref{sec_proofs}.

\begin{prop}\label{prop1}
The process $\xi_{\mathbb{Q}}$ is max-stable and stationary and has
finite-dimensional distributions given by \eqref{fdd} with $\eta$
induced by $\mathbb{Q}$.
\end{prop}

This new class of processes thus also realize a large variety of
extremal dependence structures,
which can for instance be measured by the \emph{extremal correlation
function} $\rho$  \cite{sch2003,str2012}.
For a stationary, max-stable random field $\{X(t),t\in\R^d\}$ with
Gumbel margins, $\rho$ is a natural approach to measure bivariate
extremal dependencies and for $h\in\R^d$ it is determined by
\begin{eqnarray*}
\P\bigl(X(0) \leq x, X(h) \leq x\bigr) = \P\bigl(X(0) \leq x
\bigr)^{2-\rho(h)},
\end{eqnarray*}
for some (and hence all) $x\in\R$. For instance, for the process in
\eqref{BRproc} it is given by
\begin{eqnarray*}
\rho_\gamma(h) = 2 \bigl(1 - \Phi \bigl( \sqrt{\gamma(h)}/2 \bigr)
\bigr), \qquad h\in\R^d.
\end{eqnarray*}
The processes introduced in Definition~\ref{def1} extend this class of
extremal correlation functions. Indeed, for an arbitrary variogram
$\gamma$ and mixture measure $\nu$ on $(0,\infty)$, let the measure
$\mathbb{Q}$\vadjust{\goodbreak} in Definition~\ref{def1} be the law of the scale mixture
$S^2\gamma$, where $S$ is $\nu$-distributed. The corresponding
process $\xi_{\mathbb{Q}}$ possesses the extremal correlation function
%
\begin{eqnarray}
\label{mix_ecf} \rho_{\gamma,\nu}(h) = \int_0^\infty2
\bigl(1 - \Phi \bigl( s\sqrt{\gamma(h)} \bigr) \bigr) \nu(\mathrm{d}s), \qquad h\in
\R^d.
\end{eqnarray}
Moreover, from the construction it is obvious that processes with this
dependence structure can be simulated easily as max-mixtures of
Brown--Resnick processes.
Gneiting \cite{gne1999} analyzes this kind of scale
mixtures of the complementary error function in a more general
framework. The following corollary is a consequence of Theorems 3.7 and
3.8 therein.

\begin{cor}
\label{prop_gne}
For a fixed variogram $\gamma$ the class of extremal correlation
functions in \eqref{mix_ecf} is given by all functions
$\varphi(\sqrt {\gamma(h)}), h\in\R^d$, where $\varphi\dvtx [0,\infty)\to\R$ is a
continuous function with $\varphi(0) = 1$, $\lim_{h\to\infty}
\varphi(h) = 0$, and the function
%
\begin{eqnarray}
\label{compl_mon} (-1)^k \frac{ \mathrm{d}^k}{\mathrm{d}h^k} \bigl[-
\varphi'(\sqrt{h})\bigr]
\end{eqnarray}
is nonnegative for infinitely many positive integers $k$, i.e.,
$-\varphi'(\sqrt{h})$ is completely monotone
(cf. the paragraph after Theorem~3.8 in Gneiting \cite{gne1999}).
\end{cor}

For instance, if $\nu_1$ is the Rayleigh distribution \eqref
{f_rayleigh} with density $f_1$, we obtain
\begin{eqnarray*}
\rho_{\gamma,\nu_1}(h) = 2 \biggl(1 - \int_0^\infty
\Phi (\lambda )f_{\sqrt{\gamma(h)}} (\lambda ) \,\mathrm{d}\lambda \biggr) = 1 - \biggl(
\frac{\gamma(h)}{\gamma(h) + 1} \biggr)^{1/2},\qquad  h\in\R^d,
\end{eqnarray*}
immediately from equation \eqref{rayleigh_mix}. In fact, $\rho
_{\gamma,\nu_1}(h) = \psi(\gamma(h))$,
where $\psi(x) = 1 -  (x/(x + 1)  )^{1/2}$ is a completely
monotone member of the Dagum family  \cite{ber2008}.
However, it is interesting to note that when writing $\rho_{\gamma
,\nu_1}(h) = \varphi(\sqrt{\gamma(h)})$ with
$\varphi(x) = 1 -  (x^2/(x^2 + 1)  )^{1/2}$ as in Corollary~\ref{prop_gne}, the function $\varphi$
merely satisfies \eqref{compl_mon} but is not completely monotone.

Similarly, for the Type-2 Gumbel distribution with $b=1$, the extremal
correlation function
is given by $\rho(h) = \exp( -\sqrt{ 2\gamma(h) })$. In particular,
it follows that for any
variogram $\gamma$ and any $r>0$ the function
\begin{eqnarray*}
\rho(h) = \exp \bigl( -r\sqrt{ \gamma(h) } \bigr), \qquad h\in\R^d,
\end{eqnarray*}
is an extremal correlation function. Since this class of extremal
correlation functions is closed under the operation of mixing with
respect to probability measures, this implies that for any measure $\mu
\in\mathcal{M}_1((0,\infty))$ the Laplace transform $\mathcal{L}
\mu$ yields an extremal correlation function
\begin{eqnarray*}
\rho_\mu(h) = \mathcal{L} \mu\bigl(\sqrt{ \gamma(h) }\bigr) = \int
_0^\infty \mathrm{e}^{ -r\sqrt{ \gamma(h) }} \mu(\mathrm{d}r),\qquad  h\in
\R^d.
\end{eqnarray*}
Equivalently, for any completely monotone function $\psi$ with $\psi
(0) = 1$, the function $\psi(\sqrt{ \gamma(h) })$ is an extremal
correlation function. A corresponding max-stable, stationary random
field is given by a max-mixture of Brown--Resnick processes with
suitable $\nu\in\mathcal{M}_1((0,\infty))$.

\section{Proofs}
\label{sec_proofs}
\begin{pf*}{Proof of Theorem~\ref{thm1}}
Let $x,y\in\mathbb{R}$ and put $u_n(z) = b_n + z/b_n$, for $z\in
\mathbb{R}$.
%
\begin{eqnarray}\label{eq:01}
&&\log\P \Bigl( \max_{i=1,\dots,n} X^{(1)}_{i,n}
\leq u_n(x), \max_{i=1,\dots,n} X^{(2)}_{i,n}
\leq u_n(y) \Bigr)\nonumber
\\
&&\quad  = \sum_{i=1}^n \log \bigl( 1 -
\bigl[\P\bigl( X^{(1)}_{i,n} > u_n(x)\bigr) + \P
\bigl( X^{(2)}_{i,n} > u_n(y)\bigr) - \P \bigl(
X^{(1)}_{i,n} > u_n(x), X^{(2)}_{i,n}
> u_n(y) \bigr) \bigr] \bigr)
\nonumber\\[-8pt]\\[-8pt]
&&\quad = - \sum_{i=1}^n \P\bigl(
X^{(1)}_{i,n} > u_n(x)\bigr) - \sum
_{i=1}^n \P\bigl( X^{(2)}_{i,n} >
u_n(y)\bigr)\nonumber
\\
&&\qquad {} + \sum_{i=1}^n \P \bigl(
X^{(1)}_{i,n} > u_n(x), X^{(2)}_{i,n}
> u_n(y) \bigr) + R_n,\nonumber
\end{eqnarray}
where $R_n$ is a remainder term from the Taylor expansion of $\log
(1-z) = -z - z^2/2 + \mathrm{o}(z^2)$, as $z\to0$. Thus, by \eqref{asymp_neg}
there is an $n_0\in\N$ s.t. for all $n\geq n_0$ we have
%
\begin{eqnarray}\label{Rn}
|R_n| &\leq&\sum_{i=1}^n
\bigl[\P\bigl( X^{(1)}_{i,n} > u_n(x)\bigr) + \P
\bigl( X^{(2)}_{i,n} > u_n(y)\bigr)
\bigr]^2\nonumber
\\
& \leq&\max_{i=1,\ldots, n} \bigl[\P\bigl(
X^{(1)}_{i,n} > u_n(x)\bigr) + \P\bigl(
X^{(2)}_{i,n} > u_n(y)\bigr) \bigr]
\\
&&{} \cdot\sum_{i=1}^n \bigl[\P
\bigl( X^{(1)}_{i,n} > u_n(x)\bigr) + \P\bigl(
X^{(2)}_{i,n} > u_n(y)\bigr) \bigr].\nonumber
\end{eqnarray}
For the one-dimensional margins, we observe
\begin{eqnarray*}
- \sum_{i=1}^n \P \bigl(
X^{(1)}_{i,n} > u_n(x) \bigr) &=& - \sum
_{i=1}^n\int_{u_n(x)/\sigma_{i,n,1}}^\infty
\phi(z) \,\mathrm{d}z
\\
& =& - \sum_{i=1}^n \int
_{x/\sigma_{i,n,1} - b_n^2(1 - 1/\sigma
_{i,n,1})}^\infty \frac{1}{b_n} \phi
\bigl(u_n(z)\bigr) \,\mathrm{d}z
\\
& =& - \int_{[0,\infty]\times\mathbb{R}^2} \int_{(1 - \theta
/b_n^2)x - \theta}^\infty
\mathrm{e}^{-z-z^2/(2b_n^2)} \,\mathrm{d}z \eta_n\bigl(\mathrm{d}(\lambda ,\theta,\gamma)\bigr),
\end{eqnarray*}
where for the last equation we used $b_n=n\phi(b_n)$ and the
definition of the measure $\eta_n$ in \eqref{def_eta_n} to replace
the sum
by the integral. For $n\in\N$, let
\begin{eqnarray*}
h_n(\theta) = \int_{(1 - \theta/b_n^2)x - \theta}^\infty
\mathrm{e}^{-z-z^2/(2b_n^2)} \,\mathrm{d}z, \qquad \theta\in\R.
\end{eqnarray*}
Clearly, as $n\to\infty$, $h_n$ converges uniformly on compact sets
to the function $h(\theta) = \exp(\theta- x)$. Note that $h$ and
$h_n$ are continuous functions on $\R$. Put $\omega= (\lambda,\theta
,\gamma)$ and observe for $K>0$ that
%
\begin{eqnarray}\label{hn_conv}
&&\biggl\llvert \int_{[0,\infty]\times\mathbb{R}^2} h_n(\theta)
\eta_n(\mathrm{d}\omega) - \int_{[0,\infty]\times\mathbb{R}^2} h(\theta) \eta(\mathrm{d}
\omega)\biggr\rrvert
\nonumber\\
&&\quad  \leq\biggl\llvert \int_{[0,\infty]\times\mathbb{R}^2} h_n(
\theta)\mathbf{1}_{h_n > K} \eta_n(\mathrm{d}\omega) - \int
_{[0,\infty
]\times\mathbb{R}^2} h(\theta)\mathbf{1}_{h > K} \eta(\mathrm{d}\omega )\biggr
\rrvert
\\
&&\qquad {} + \biggl\llvert \int_{[0,\infty]\times\mathbb{R}^2} h_n(
\theta)\mathbf{1}_{h_n < K} \eta_n(\mathrm{d}\omega) - \int
_{[0,\infty
]\times\mathbb{R}^2} h(\theta)\mathbf{1}_{h < K} \eta(\mathrm{d}\omega )\biggr
\rrvert .\nonumber
\end{eqnarray}
By Theorem~5.5 in Billingsley \cite{bil1968} (see also the
remark after the theorem), $\eta_n h_n^{-1}$ converges weakly to $\eta
h^{-1}$. Moreover, since $h\mathbf{1}_{h < K}$ and the $h_n\mathbf
{1}_{h_n < K}$ are uniformly bounded in $n$, the second summand in
\eqref{hn_conv} converges to $0$ as $n\to\infty$, for arbitrary
$K>0$. By the integrability condition \eqref{unif_int} and Fatou's
lemma, we have $\int_{[0,\infty]\times\mathbb{R}^2} h(\theta) \eta
(\mathrm{d}\omega) < \infty$ and hence, also the first summand in \eqref
{hn_conv} tends to zero as $K,n\to\infty$. Consequently,
%
\begin{eqnarray}
\label{eq:02} - \sum_{i=1}^n \P \bigl(
X_{i,n}^{(1)} > u_n(x) \bigr) \to- \int
_{[0,\infty]\times\mathbb{R}^2} \exp \bigl[- (x - \theta ) \bigr] \eta(\mathrm{d}\omega).
\end{eqnarray}
Similarly, we get
%
\begin{eqnarray}
\label{eq:03} - \sum_{i=1}^n \P \bigl(
X_{i,n}^{(2)} > u_n(y) \bigr) \to- \int
_{[0,\infty]\times\mathbb{R}^2} \exp \bigl[- (y - \gamma ) \bigr] \eta(\mathrm{d}\omega).
\end{eqnarray}
It now also follows from \eqref{asymp_neg}, \eqref{Rn}, \eqref
{eq:02} and \eqref{eq:03} that the remainder term $R_n$ converges to
zero as $n\to\infty$.

We now turn to the third term in \eqref{eq:01}.
\begin{eqnarray*}
&&\sum_{i=1}^n \P\bigl(
X^{(1)}_{i,n}/\sigma_{i,n,1} > u_n(x)/\sigma
_{i,n,1}, X^{(2)}_{i,n}/\sigma_{i,n,2} >
u_n(y)/\sigma_{i,n,2}\bigr)
\\
&&\quad  =\sum_{i=1}^n \int_{u_n(y)/\sigma_{i,n,2}}^\infty
\biggl[ 1 - \Phi \biggl(\frac{u_n(x)/\sigma_{i,n,1} - \rho_{i,n} z}{(1-\rho
_{i,n}^2)^{1/2}} \biggr) \biggr] \phi(z) \,\mathrm{d}z
\\
&&\quad  =\frac{1}n \sum_{i=1}^n \int
_{y/\sigma_{i,n,2} - b_n^2
(1-1/\sigma_{i,n,2} )}^\infty \biggl[ 1 - \Phi \biggl(
\frac
{u_n(x)/\sigma_{i,n,1} - \rho_{i,n}u_n(z) }{(1-\rho_{i,n}^2)^{1/2}} \biggr) \biggr] \mathrm{e}^{-z-z^2/(2b_n^2)} \,\mathrm{d}z
\\
&&\quad  =\int_{[0,\infty]\times\R^2} \int_{(1-\gamma/b_n^2)y - \gamma
}^\infty
\bigl[ 1 - \Phi \bigl( s_n(\lambda, \theta, z, x) \bigr) \bigr]
\mathrm{e}^{-z-z^2/(2b_n^2)} \,\mathrm{d}z \eta_n(\mathrm{d}\omega),
\end{eqnarray*}
where we used $b_n = n \phi(b_n)$ for the second last equation and
$s_n$ is defined by
\begin{eqnarray*}
s_n(\lambda, \theta, z, x) := \frac{\lambda}{(1-\lambda^2/b_n^2)^{1/2}} +
\frac{(1-\theta
/b_n^2)x-z-\theta}{(1-\lambda^2/b_n^2)^{1/2}2\lambda} + \frac
{\lambda z}{(1-\lambda^2/b_n^2)^{1/2}b_n^2}.
\end{eqnarray*}
For the last equation, we replaced the sum by the integral w.r.t. the
empirical measure $\eta_n$ as in \eqref{def_eta_n}. Note that for
$i\in\{1,\dots, n\}$, in fact a short computation yields
\begin{eqnarray*}
s_n \Bigl(\sqrt{b_n^2(1-
\rho_{i,n})/2}, b_n^2(1-1/\sigma_{i,n,1}),
z, x \Bigr) = \frac{u_n(x)/\sigma_{i,n,1} - \rho_{i,n}u_n(z)
}{(1-\rho_{i,n}^2)^{1/2}}.
\end{eqnarray*}
For $n\in\N$, let
\begin{eqnarray*}
g_n(\lambda, \theta,\gamma) = \mathbf{1}_{\lambda\leq b_n}\int
_{(1-\gamma/b_n^2)y - \gamma}^\infty \bigl[ 1 - \Phi \bigl(
s_n(\lambda, \theta, z, x) \bigr) \bigr] \mathrm{e}^{-z-z^2/(2b_n^2)} \,\mathrm{d}z
\end{eqnarray*}
be a measurable function on $[0,\infty]\times\R^2$. It is easy to
see, that as $n\to\infty$, $g_n$ converges pointwise to the function
\begin{eqnarray*}
g(\lambda,\theta,\gamma) = \int_{y-\gamma}^\infty \bigl[
1 - \Phi \bigl( s(\lambda, \theta, z, x) \bigr) \bigr] \mathrm{e}^{-z} \,\mathrm{d}z,
\end{eqnarray*}
with
\begin{eqnarray*}
s(\lambda, \theta, z, x ) := \lambda+ \frac{x-z-\theta}{2\lambda}.
\end{eqnarray*}
Note that $g$ is a continuous function on $[0,\infty]\times\R^2$ and
$g(0,\theta,\gamma) = g_n(0,\theta,\gamma) =\linebreak[4]  \exp({-\max(x-\theta
,y-\gamma)})$ and $g(\infty,\theta,\gamma) = g_n(\infty,\theta
,\gamma) = 0$, for any $(\theta,\gamma)\in\R^2$ and $n$
sufficiently large.
Here, the values are understood as the limits as $\lambda\to0$ and
$\lambda\to\infty$ (using dominated convergence), respectively, for example,
$\lim_{\lambda\to0} g(\lambda,\theta,\gamma) = \int_{y-\gamma
}^\infty\mathbf{1}_{z>x-\theta} \mathrm{e}^{-z} \,\mathrm{d}z = \exp({-\max(x-\theta
,y-\gamma)})$.
In order to establish the weak convergence $\eta_n g_n^{-1}
\Rightarrow\eta g^{-1}$, we show that $g_n$ converges uniformly on
compact sets to $g$ as $n\to\infty$. To this end, let $C = [0,\infty
]\times[\theta_0,\theta_1]\times[\gamma_0,\gamma_1]$ be an
arbitrary compact set in $[0,\infty]\times\R^2$ and let $\epsilon
>0$ be given. First, note that instead of $g_n$ it suffices to consider
the function $\tilde{g}_n$, defined as
\begin{eqnarray*}
\tilde{g}_n(\lambda, \theta,\gamma) = \mathbf{1}_{\lambda\leq
b_n}\int
_{(1-\gamma/b_n^2)y - \gamma}^\infty \bigl[ 1 - \Phi \bigl(
s_n(\lambda, \theta, z, x) \bigr) \bigr] \mathrm{e}^{-z} \,\mathrm{d}z,
\end{eqnarray*}
since for $n$ large enough
\begin{eqnarray*}
\sup_{(\lambda,\theta,\gamma)\in C} | g_n(\lambda, \theta,\gamma ) -
\tilde{g}_n(\lambda, \theta,\gamma)| \leq\mathbf{1}_{\lambda
\leq b_n}
\int_{-2|y| - \gamma_1}^\infty \mathrm{e}^{-z}
\bigl(1-\mathrm{e}^{-z^2/(2b_n^2)}\bigr)\,\mathrm{d}z \to0,
\end{eqnarray*}
as $n \to\infty$, by dominated convergence. Further, for any
$\epsilon> 0$, let $z_1 > -\log\epsilon$ which implies $\int_{z_1}^\infty \mathrm{e}^{-z} \,\mathrm{d}z < \epsilon$. We note that for $n$ large enough
\begin{eqnarray*}
s_n(\lambda, \theta, z, x) &\geq& \bigl(1-\lambda^2/b_n^2
\bigr)^{-1/2} \biggl( \lambda \biggl(1 + \frac{-2|y| - \gamma_1}{b_n^2} \biggr) +
\frac{-2|x|-z_1-\theta_1}{2\lambda} \biggr)
\\
& \geq& \biggl( \frac\lambda2 + \frac{-2|x|-z_1-\theta_1}{2\lambda
} \biggr),
\end{eqnarray*}
for all $\lambda\leq b_n$, $-2|y| - \gamma_1\leq z\leq z_1$ and
$(\lambda,\theta,\gamma)\in C$, independently of $n\in\N$. Hence,
there is a $\lambda_1>0$ s.t. for all $\lambda_1\leq\lambda\leq b_n$
\[
1 - \Phi \bigl( s_n(\lambda, \theta, z, x) \bigr) < \epsilon
\mathrm{e}^{-2|y| - \gamma_1}.
\]
Thus, for all $n\in\N$ large enough,
\begin{eqnarray*}
\sup_{(\lambda,\theta,\gamma)\in C, \lambda\geq\lambda_1 } \tilde{g}_n(\lambda, \theta,\gamma)
\leq\mathbf{1}_{\lambda\leq
b_n} \biggl(\int_{-2|y| - \gamma_1}^{z_1}
\epsilon \mathrm{e}^{-2|y| - \gamma
_1}\mathrm{e}^{-z} \,\mathrm{d}z + \int_{z_1}^\infty
\mathrm{e}^{-z} \,\mathrm{d}z \biggr) \leq2\epsilon,
\end{eqnarray*}
and in the same manner, $ \sup_{(\lambda,\theta,\gamma)\in C,
\lambda\geq\lambda_1 } g(\lambda, \theta,\gamma) \leq2\epsilon
$. Furthermore, we observe
\begin{eqnarray*}
\lim_{\lambda\to0} \Phi \bigl( s_n(\lambda, \theta, z,
x) \bigr) = \mathbf{1}_{z < (1-\theta/b_n^2)x - \theta} \quad \mbox{and}\quad  \lim_{\lambda\to0}
\Phi \bigl(s(\lambda, \theta, z, x) \bigr) = \mathbf{1}_{z < x - \theta}.
\end{eqnarray*}
Choose $n_0\in\N$ such that for all $n > n_0$ and all $\theta\in
[\theta_0,\theta_1]$ we find an open interval $(a_{\theta},
b_{\theta})$ of size $\epsilon/2$ that contains $\{(1-\theta/b_n^2)x
- \theta, x - \theta\}$. Put $I_{\theta} = (a_{\theta} - \epsilon
/4, b_{\theta}+ \epsilon/4)$, then we find a $\lambda_0 > 0 $, s.t.
for all $(\lambda,\theta,\gamma)\in C,\lambda\leq\lambda_0$,
$z\in I_{\theta}$ and $n > n_0$, we have
$|\Phi ( s_n(\lambda, \theta, z, x)  ) - \Phi
(s(\lambda, \theta, z, x)  )| \leq\epsilon$. Consequently,
\begin{eqnarray*}
&&\sup_{(\lambda,\theta,\gamma)\in C, \lambda\leq\lambda_0 } \bigl|\tilde{g}_n(\lambda, \theta,\gamma)
- g(\lambda, \theta,\gamma )\bigr|
\\
&&\quad \leq\sup_{(\lambda,\theta,\gamma)\in C, \lambda\leq\lambda_0 } \int_{-2|y| - \gamma_1}^\infty
(\mathbf{1}_{z\in I_\theta} + \epsilon\mathbf{1}_{z\in\R\setminus I_\theta} )
\mathrm{e}^{-z}\,\mathrm{d}z \leq2\epsilon \mathrm{e}^{2|y| + \gamma_1}.
\end{eqnarray*}
Choose $n_1\in\N$, s.t. $b_{n_1} > \lambda_1$. For $\lambda_0\leq
\lambda\leq\lambda_1$ and $n>n_1$,
%
\begin{eqnarray}
\label{eq2}&&\bigl| s_n(\lambda, \theta, z, x) - s(\lambda, \theta, z,
x) \bigr|
\nonumber\\
&&\quad = \biggl\llvert \biggl(\lambda+ \frac{x - z - \theta}{ 2
\lambda} \biggr) \biggl(1 -
\frac{1}{(1-\lambda
_1^2/b_n^2)^{1/2}} \biggr) - \frac{\lambda^2 z - \theta}{(1-\lambda
_1^2/b_n^2)^{1/2}b_n^22\lambda} \biggr\rrvert
\\
&&\quad  \leq M_1 \biggl\llvert 1 - \frac{1}{(1-\lambda
_0^2/b_n^2)^{1/2}}\biggr
\rrvert + \frac{M_2}{(1-\lambda
_1^2/b_n^2)^{1/2}b_n^2} \to0\nonumber
\end{eqnarray}
for $n\to\infty$, uniformly in $z\in[-2|y| - \gamma_1,z_1]$ and
$(\lambda,\theta,\gamma)\in C$ with $\lambda_0\leq\lambda\leq
\lambda_1$. Here, $M_1$ and $M_2$ are positive constants that only
depend on $x,y,\lambda_0,\lambda_1,\theta_0,\theta_1,\gamma_1$.
Let $n_2\in\N$, s.t. for all $n>\max(n_1,n_2)$ the difference in
\eqref{eq2} is less than or equal to $\epsilon \mathrm{e}^{-2|y| - \gamma_1}$.
By the Lipschitz continuity of $\Phi$, we obtain for all $\lambda
_0\leq\lambda\leq\lambda_1$ and $(\lambda,\theta,\gamma)\in C$,
\begin{eqnarray*}
&&\int_{-2|y| - \gamma_1}^\infty\bigl\llvert \Phi \bigl(
s_n(\lambda, \theta, z, x) \bigr) - \Phi \bigl( s(\lambda, \theta,
z, x) \bigr) \bigr\rrvert \mathrm{e}^{-z} \,\mathrm{d}z
\\
&&\quad  \leq\int_{-2|y| - \gamma_1}^{z_1} \bigl\llvert
s_n(\lambda, \theta, z, x) - s(\lambda, \theta, z, x) \bigr\rrvert
\mathrm{e}^{-z} \,\mathrm{d}z + \int_{z_1}^\infty
\mathrm{e}^{-z} \,\mathrm{d}z
\\
&&\quad  \leq\int_{-2|y| - \gamma_1}^{z_1} \epsilon \mathrm{e}^{-2|y| - \gamma_1}
\mathrm{e}^{-z} \,\mathrm{d}z + \int_{z_1}^\infty
\mathrm{e}^{-z} \,\mathrm{d}z \leq2\epsilon.
\end{eqnarray*}
Putting the parts together yields
\begin{eqnarray*}
\lim_{n\to\infty} \sup_{(\lambda,\theta,\gamma)\in C} \bigl|\tilde
{g}_n(\lambda,\theta,\gamma) - g(\lambda,\theta,\gamma)\bigr| = 0.
\end{eqnarray*}
The assumptions of Theorem~5.5 in Billingsley \cite{bil1968}
are satisfied and therefore $\eta_n g_n^{-1}$ converges weakly to
$\eta g^{-1}$. By a similar argument as in \eqref{hn_conv} together
with the integrability condition \eqref{unif_int}, we obtain for $n\to
\infty$
\begin{eqnarray*}
\sum_{i=1}^n \P\bigl(
X^{(1)}_{i,n} > u_n(x), X^{(2)}_{i,n}
> u_n(y)\bigr) \to \int_{[0,\infty]\times\R^2} g(\lambda,\theta,
\gamma) \eta \bigl(\mathrm{d}(\lambda,\theta,\gamma)\bigr).
\end{eqnarray*}
Finally, partial integration gives
\begin{eqnarray*}
g(\lambda,\theta,\gamma) &=& \mathrm{e}^{-(y-\gamma)} + \mathrm{e}^{-(x-\theta)} - \Phi
\biggl(\lambda+\frac{y-x + \theta- \gamma}{2\lambda} \biggr)\mathrm{e}^{-(x-\theta)}
\\
&&{}- \Phi \biggl(\lambda-\frac{y-x + \theta- \gamma}{2\lambda} \biggr)\mathrm{e}^{-(y-\gamma)}.
\end{eqnarray*}
Together with \eqref{eq:01}, \eqref{eq:02}, \eqref{eq:03} and the
fact that $R_n$ converges to zero, this implies the desired result.
\end{pf*}
\begin{pf*}{Proof of Theorem~\ref{thm2}}
Sufficiency is a simple consequence of Theorem~\ref{thm1}, where the
covariance matrix of $\mathbf{X}_{i,n}$ is given by
\[
\pmatrix{ 1 & \rho_{i,n}
\cr
\rho_{i,n} & 1 }. %
\]
For necessity, suppose that the sequence $(\max_{i=1,\dots,n}
b_n(\mathbf{X}_{i,n} - b_n))_{n\in\N}$ of bivariate random vectors
converges in distribution to some random vector $Y$. Let the $\nu_n$,
$n\in\N$, be defined as in \eqref{emp_meas} and assume that the
sequence $ ( \nu_n )_{n\in\N}\subset\mathcal
{M}_1([0,\infty])$ does not converge. Then, by sequential compactness,
it has at least two different accumulation points $\nu,\tilde{\nu
}\in\mathcal{M}_1([0,\infty])$. By the first part of this theorem,
$ (\max_{i=1,\dots,n} b_n(\mathbf{X}_{i,n} - b_n) )_{n\in
\N}$ converges in distribution to $F_{\nu}\equiv F_{\tilde{\nu}}$.
It now suffices to show that $F_{\nu}\equiv F_{\tilde{\nu}}$ implies
$\nu\equiv\tilde{\nu}$ to conclude that $ ( \nu_n
)_{n\in\N}\subset\mathcal{M}_1([0,\infty])$ converges to some
measure $\nu$ and that $Y$ has distribution $F_{\nu}$.

The fact that there is a one-to-one correspondence between H\"
usler--Reiss distributions
$F_\lambda$ and the dependence parameter $\lambda\in[0,\infty]$ is
straightforward  \cite{kab2009}.
Showing a similar result in our case, however, requires more effort.

To this end, for two measures $\nu_1,\nu_2\in\mathcal
{M}_1([0,\infty])$ define random
variables $Y_1$ and $Y_2$ with distribution $F_{\nu_1}$ and $F_{\nu
_2}$, respectively.
First, suppose that $\nu_1(\{\infty\}) = \nu_2(\{\infty\}) = 0$.
For $j=1,2$, by Remark~\ref{rem2}
we have the stochastic representation $Y_j = \max_{i\in\N}
(U_{i,j},U_{i,j} + B_{i,j})$, where
$\sum_{i=1}^\infty\delta_{U_{i,j}}$ are Poisson point process on $\R
$ with intensity $\mathrm{e}^{-u}\,\mathrm{d}u$
and the $(B_{i,j})_{i\in\N}$ are i.i.d. copies of the random variable
$B_j$ with normal distribution
$N(-2S_j^2,4S_j^2)$, where $S_j$ is $\nu_j$-distributed. Assume that
%
\begin{eqnarray}
\label{F_equal} F_{\nu_1}(x,y) = F_{\nu_2}(x,y), \qquad \mbox{for all }
x,y\in\R,
\end{eqnarray}
that is, the max-stable distributions of $Y_1$ and $Y_2$ are equal.
Since a Poisson point process is determined by its intensity on a
generating system of the $\sigma$-algebra, it follows that the point
processes $\Pi_1 = \sum_{i=1}^\infty\delta_{(U_{i,1},U_{i,1} +
B_{i,1})}$ and $\Pi_2=\sum_{i=1}^\infty\delta_{(U_{i,2},U_{i,2} +
B_{i,2})}$ are equal in distribution. Therefore, the measurable mapping
\[
h\dvtx  \R^2 \to\R^2, \qquad (x_1,x_2)
\mapsto(x_1, x_2 -x_1)
\]
induces two Poisson point processes $h(\Pi_1)$ and $h(\Pi_2)$ on $\R
^2$ with coinciding intensity measures $\mathrm{e}^{-u}\,\mathrm{d}u \P_{B_1}(\mathrm{d}x)$ and
$\mathrm{e}^{-u}\,\mathrm{d}u \P_{B_2}(\mathrm{d}x)$, respectively. Hence, $B_1$ and $B_2$ have the
same distribution. Denote by $\psi_j$ the Laplace transform of the
Gaussian mixture $B_j$, $j=1,2$. A straightforward calculation yields
for $u\in(0,1)$
\begin{eqnarray*}
\psi_j(u) = \E\exp( u B_j ) = \int
_{[0,\infty)} \exp\bigl( -2\lambda ^2\bigl( u -
u^2 \bigr) \bigr) \nu_j(\mathrm{d}\lambda),\qquad  j=1,2.
\end{eqnarray*}
By Lemma~7 in Kabluchko \textit{et al.} \cite{kab2009}, this implies the
equality of measures $\nu^2_1(\mathrm{d}\lambda) = \nu^2_2(\mathrm{d}\lambda)$, where
$\nu^2_j$ is the image
measure of $\nu_j$ under the transformation $[0,\infty]\to[0,\infty
]$, $\lambda\mapsto\lambda^2$, for $j=1,2$. Hence, it also holds
that $\nu_1 \equiv\nu_2$.

For arbitrary $\nu_1,\nu_2\in\mathcal{M}_1([0,\infty])$, we first
need to show that $\nu_1(\{\infty\}) = \nu_2(\{\infty\})$. For
$j=1,2$, observe that for $n\in\N$
\begin{eqnarray*}
&&-\log F_{\nu_j}(-n,0) +\log F_{\nu_j}(-n,n)
\\
&&\quad  = \int_{[0,\infty)} \Phi \biggl(\lambda+
\frac
{n}{2\lambda} \biggr)\mathrm{e}^{n} + \Phi \biggl(\lambda-
\frac{n}{2\lambda
} \biggr) - \Phi \biggl(\lambda+ \frac{n}{\lambda}
\biggr)\mathrm{e}^{n} - \Phi \biggl(\lambda- \frac{n}{\lambda}
\biggr)\mathrm{e}^{-n} \nu_j(\mathrm{d}\lambda )
\\
&&\qquad {} + \bigl(1-\mathrm{e}^{-n}\bigr)\nu_j\bigl(\{\infty\}
\bigr).
\end{eqnarray*}
Since the second derivative of $\Phi$ is negative on the positive real
line, we have the estimate
\begin{eqnarray*}
\mathrm{e}^n\biggl\llvert \Phi \biggl(\lambda+ \frac{n}{2\lambda} \biggr)
- \Phi \biggl(\lambda+ \frac{n}{\lambda} \biggr)\biggr\rrvert \leq
\frac
{n}{2\lambda\sqrt{2\uppi}} \mathrm{e}^n \mathrm{e}^{-(\lambda+ n/(2\lambda))^2/2},
\end{eqnarray*}
where the latter term converges pointwise to zero as $n\to\infty$.
Moreover, it is uniformly bounded in $n\in\N$ and $\lambda\in
[0,\infty)$ by a constant and hence, by dominated convergence
\begin{eqnarray*}
\lim_{n\to\infty} -\log F_{\nu_j}(-n,0) +\log
F_{\nu_j}(-n,n) = \nu_j\bigl(\{\infty\}\bigr),\qquad  j=1,2.
\end{eqnarray*}
It therefore follows from \eqref{F_equal} that $\nu_1(\{\infty\}) =
\nu_2(\{\infty\})$. If $\nu_1(\{\infty\}) < 1$, we apply the above
to the restricted probability measures $\nu_j( \cdot \cap[0,\infty
))/(1-\nu_j(\{\infty\}))$ on $[0,\infty)$, $j=1,2$, to obtain $\nu
_1\equiv\nu_2$.

The last claim of the theorem follows from the fact that the integrand
in \eqref{distribution2} is bounded and continuous in $\lambda$ for
fixed $x,y\in\R$, and hence, for $\nu,\nu_n\in\mathcal
{M}_1([0,\infty])$, $n\in\N$, weak convergence of $\nu_n$ to $\nu$
ensures the pointwise convergence of the distribution functions.
\end{pf*}
\begin{pf*}{Proof of Corollary~\ref{cor1}}
The first statement is a consequence of Theorem~\ref{thm2}, because
every sequence of random vectors can be understood as a triangular
array where the columns contain equal random vectors.

For the second claim, let $\nu\in\mathcal{M}_1([0,\infty])$ be an
arbitrary probability measure. Similarly as in Example~\ref{rem1},
define an i.i.d. sequence $(R_i)_{i\in\N}$ of samples of $\nu$.
Choosing $\rho_{i} = \max(1 - 2R_i^2/b_i^2,-1)$ as correlation of
$\mathbf{X}_i$ yields
\begin{eqnarray*}
\nu_n = \frac{1}n \sum_{i=1}^n
\delta_{\mathbf{1}_{R_i < b_i} R_i
b_n/b_i + \mathbf{1}_{R_i > b_i} b_n}.
\end{eqnarray*}
First, consider the measures $\tilde{\nu}_n = \frac{1}n \sum_{i=1}^n
\delta_{R_i b_n/b_i}$, for $n\in\N$. For $y\in[0,\infty]$ with
$\nu(\{y\}) = 0$ we observe
%
\begin{eqnarray}
\label{nu_n_ab} \tilde{\nu}_n\bigl( [0,y] \bigr) = \frac{1}n
\sum_{i = 1}^n \mathbf {1}_{[0,y]}(R_ib_n/b_i).
\end{eqnarray}

Fix $\epsilon> 0$ and recall from \eqref{bn} that $b_n/\sqrt{2\log
n} \to1$ as $n\to\infty$. Hence, choose $n$ large enough such that
$i > n^{1/(1+\epsilon)^2}$ implies $b_n/b_i < 1 + \epsilon$. Let
$n_\epsilon$ denote the smallest integer which is strictly larger than
$n^{1/(1+\epsilon)^2}$, then \eqref{nu_n_ab} yields
\begin{eqnarray*}
\Biggl\llvert \tilde{\nu}_n\bigl( [0,y] \bigr) - \frac{1}n
\sum_{i = 1}^n \mathbf {1}_{[0,y]}(
R_i) \Biggr\rrvert &\leq&\frac{n_\epsilon}{n} + \frac{1}n
\Biggl\llvert \sum_{i = n_\epsilon} ^n
\mathbf{1}_{[0,y]}( R_ib_n/b_i ) -
\sum_{i =n_\epsilon}^n \mathbf{1}_{[0,y]}(
R_i) \Biggr\rrvert
\\
& \leq&\frac{n_\epsilon}{n} + \frac{1}n \sum_{i = n_\epsilon}^n
\mathbf{1}_{(y/(1+\epsilon),y]}(R_i).
\end{eqnarray*}
Letting $n\to\infty$ gives
\begin{eqnarray*}
\lim_{n\to\infty} \Biggl\llvert \tilde{\nu}_n\bigl( [0,y]
\bigr) - \frac{1}n \sum_{i = 1}^n
\mathbf{1}_{[0,y]}( R_i) \Biggr\rrvert \leq\nu\bigl((y/(1+
\epsilon ),y]\bigr), \qquad \mbox{a.s.}
\end{eqnarray*}
Since $\epsilon$ was arbitrary and $\nu(\{y\}) = 0$, it follows from
the law of large numbers that $\tilde{\nu}_n$ converges a.s. weakly
to $\nu$, as $n\to\infty$. Similarly, one can see that the sequence
$(\nu_n)_{n\in\N}$ has a.s. the same limit as $(\tilde{\nu
}_n)_{n\in\N}$, as $n\to\infty$.
\end{pf*}
\begin{pf*}{Proof of Theorem~\ref{multi_dim}}
Let $u_n(z)=b_n + z/b_n$ for $z\in\R$, $u_n(\mathbf
{x})=(u_n(x_1),\dots,u_n(x_d))^\top$ for $\mathbf{x}\in\R^d$ and
for $\mathbf{x},\mathbf{y}\in\R^d$ write $\mathbf{x} > \mathbf
{y}$ if $x_i>y_i$ for all $1\leq i\leq d$.

Let $\mathbf{x} = (x_1,\dots,x_d)^\top\in\R^d$ be a fixed vector
and $A_{i,n}^l =  \{ X^{(l)}_{i,n} \leq u_n(x_l) \}$ for
$n\in\N, 1\leq i\leq n$ and $1\leq l\leq d$.
%
\begin{eqnarray}\label{log_expan}
&&\log\P \Bigl( \max_{i=1,\dots,n} X^{(1)}_{i,n}
\leq u_n(x_1),\dots, \max_{i=1,\dots,n}
X^{(d)}_{i,n} \leq u_n(x_d) \Bigr)
\nonumber\\[-8pt]\\[-8pt]
&&\quad  = \sum_{i=1}^n \log\P
\Biggl[ \bigcap_{l=1}^d
A_{i,n}^l \Biggr] = -\sum_{i=1}^n
\P \Biggl[ \bigcup_{l=1}^d
\bigl(A_{i,n}^l \bigr)^C \Biggr] +
R_n,\nonumber
\end{eqnarray}
where $R_n$ is a remainder term from the Taylor expansion of $\log$.
Using the same arguments as for the remainder term in \eqref{Rn}, we
conclude that $R_n$ converges to zero as $n\to\infty$. By the
additivity formula we have
%
\begin{eqnarray}
\label{add_formula} -\P \Biggl[ \bigcup_{l=1}^d
\bigl(A_{i,n}^l \bigr)^C \Biggr] = \sum
_{l=1}^d (-1)^l \sum
_{ m: 1\leq m_1 < \cdots< m_l\leq d } \P \Biggl[ \bigcap_{k=1}^l
\bigl(A_{i,n}^{m_k} \bigr)^C \Biggr].
\end{eqnarray}
Consequently, by \eqref{log_expan} and \eqref{add_formula} it
suffices to show that
%
\begin{eqnarray}
\label{suff_conv} \lim_{n\to\infty} \sum_{i=1}^n
\P \bigl( \mathbf{X}_{i,n} > u_n(\mathbf{x}) \bigr) = \int
_{[0,\infty)^{d\times d}} h_{d,(1,\ldots,d),\Lambda}(x_1,\dots
,x_d) \eta(\mathrm{d}\Lambda).
\end{eqnarray}
Let $\mathbf{Z} = (Z_1,\ldots,Z_d)$ be a standard normal random
vector with independent margins and let
$K=\{1,\ldots,d-1\}$. For a vector $\mathbf{x}\in\R^d$ let $\mathbf
{x}_K = (x_1,\dots,x_{d-1})$. If
$A =  (a_{j,k} )_{1\leq j,k \leq d}\in\R^{d\times d}$ is a
matrix, let
$A_{d,K} = (a_{d,1},\dots, a_{d,d-1})$, $A_{K,d} = (a_{1,d},\dots,
a_{d-1,d})$ and $A_{K,K} = (a_{j,k})_{j,k\in K}$.

We first assume that all $\mathbf{X}_{i,n}$ are non-degenerate, that
is, $\eta_n^2(D_0) = 1$, for all $n\in\N$. Then,
similarly as in the proof of Theorem~1.1 in Hashorva \textit{et al.}
\cite{has2012a}, we define a new matrix $B_{i,n}\in\R^{(d-1)\times
(d-1)}$ by
%
\begin{eqnarray}
\label{Bin} B_{i,n}B_{i,n}^\top= (
\Sigma_{i,n})_{K,K} - \boldsymbol{\sigma }_{i,n}
\boldsymbol{\sigma}_{i,n}^\top,\qquad  \boldsymbol{
\sigma}_{i,n}= (\Sigma _{i,n})_{K,d},
\end{eqnarray}
which is well-defined since $(\Sigma_{i,n})_{K,K} - \boldsymbol{\sigma
}_{i,n}\boldsymbol{\sigma}_{i,n}^\top$ is positive definite as the Schur
complement of $(\Sigma_{i,n})_{d,d}$ in the positive definite matrix
$\Sigma_{i,n}$. This enables us to write the vector $\mathbf
{X}_{i,n}$ as the joint stochastic representation
\begin{eqnarray*}
\bigl(X^{(1)}_{i,n},\dots,X^{(d-1)}_{i,n}
\bigr) \stackrel{d} {=} B_{i,n} \mathbf{Z}_K +
Z_d \boldsymbol{\sigma}_{i,n}, \qquad X^{(d)}_{i,n}
\stackrel{d} {=} Z_d.
\end{eqnarray*}
Therefore, since $Z_d$ is independent of $\mathbf{Z}_K$,
%
\begin{eqnarray}\label{single_summand}
\P \bigl( \mathbf{X}_{i,n} > u_n(\mathbf{x})
\bigr) &=& \P \bigl( B_{i,n} \mathbf{Z}_K + Z_d
\boldsymbol{\sigma}_{i,n} > u_n(\mathbf{x}_K),
Z_d > u_n(x_d) \bigr)
\nonumber\\
& =& \int_{x_d}^\infty\P \bigl(
B_{i,n}\mathbf{Z}_K + u_n(s)\boldsymbol{
\sigma}_{i,n} > u_n(\mathbf{x}_K) \bigr)
b_n^{-1} \phi(b_n)\mathrm{e}^{- s - s^2/(2b_n^2) } \,\mathrm{d}s
\nonumber\\
& =& \frac{1}n \int_{x_d}^\infty S
\bigl( \bigl(b_n^2\bigl(\mathbf {11}^\top-
\Sigma_{i,n}\bigr) \bigr)_{K,d} + x_K - s
\mathbf{1} \\
&&\hphantom{\frac{1}n \int_{x_d}^\infty S
\bigl( }{}+ sb_n^{-2} \bigl(b_n^2
\bigl(\mathbf{11}^\top-\Sigma_{i,n}\bigr)
\bigr)_{K,d} | b_n^2B_{i,n}B^\top_{i,n}
\bigr)
\nonumber\\
&&\hphantom{\frac{1}n \int_{x_d}^\infty  }{}\times \mathrm{e}^{- s - s^2/(2b_n^2) } \,\mathrm{d}s.\nonumber
\end{eqnarray}
It follows from the definition of $B_{i,n}$ in equation \eqref{Bin} that
\begin{eqnarray*}
B_{i,n}B^\top_{i,n} &=& \bigl(\mathbf{11}^\top-
\Sigma _{i,n}\bigr)_{K,d}\mathbf{1}^\top+
\mathbf{1}\bigl(\mathbf{11}^\top- \Sigma _{i,n}
\bigr)_{d,K} - \bigl(\mathbf{11}^\top- \Sigma_{i,n}
\bigr)_{K,K}
\\
&&{} - \bigl(\mathbf{11}^\top- \Sigma_{i,n}
\bigr)_{K,d}\bigl(\mathbf{11}^\top- \Sigma_{i,n}
\bigr)_{d,K}.
\end{eqnarray*}
Together with \eqref{single_summand} and the definition of $\eta_n$
this yields
\begin{eqnarray*}
\sum_{i=1}^n \P \bigl(
\mathbf{X}_{i,n} > u_n(\mathbf{x}) \bigr) = \int
_{D_0} p_n(A) \eta^2_n(\mathrm{d}A),
\end{eqnarray*}
where $p_n$ is a measurable function from $D_0$ to $[0,\infty)$ given by
\begin{eqnarray*}
p_n(A) &=& \int_{x_d}^\infty S \bigl(
2A_{K,d} + x_K - s\mathbf{1} + 2b_n^{-2}sA_{K,d}
| \Gamma_{d,(1,\dots,d)} (\sqrt{A}) - 4b_n^{-2}A_{K,d}A_{d,K}
\bigr)
\\
&&\hphantom{\int_{x_d}^\infty}{}\times \mathrm{e}^{- s - s^2/(2b_n^2)} \,\mathrm{d}s.
\end{eqnarray*}
Further, let $p$ be the measurable function from $D_0$ to $[0,\infty)$
\begin{eqnarray*}
p(A) = \int_{x_d}^\infty S \bigl( 2A_{K,d}
+ x_K - s\mathbf{1} | \Gamma_{d,(1,\dots,d)} (\sqrt{A})
\bigr)\mathrm{e}^{- s} \,\mathrm{d}s.
\end{eqnarray*}
Note that $\eta_n\Rightarrow\eta$ if and only if $\eta
^2_n\Rightarrow\eta^2$. In view of \eqref{suff_conv} it suffices to
show that
%
\begin{eqnarray}
\label{int_conv} \lim_{n\to\infty} \int_{D_0}
p_n(A) \eta^2_n(\mathrm{d}A) = \int
_{D_0} p(A) \eta^2(\mathrm{d}A).
\end{eqnarray}
To this end, let $A_0\in D_0$ and $\{A_n, n\in\N\}$ be a sequence in
$D_0$ that converges to $A_0$. We will show that $p_n(A_n) \to p(A_0)$
as $n\to\infty$. By dominated convergence, it is sufficient to show
the convergence of the survivor functions. Since $A_0$ is in $D_0$,
recall that $\Gamma_{d,(1,\dots,d)} (\sqrt{A_0})$ is in the space
$\mathcal{M}_{(d-1)}$ of $(d-1)$-dimensional, non-degenerate
covariance matrices. Moreover, since $\mathcal{M}_{(d-1)}\subset\R
^{(d-1)\times(d-1)}$ is open and $\Gamma_{d,(1,\dots,d)} (\sqrt {A_n}) - b_n^{-2}4(A_n)_{K,d}(A_n)_{d,K}$ converges to $\Gamma
_{d,(1,\dots,d)} (\sqrt{A_0})$, there is an $n_0\in\N$ such that
for all $n\geq n_0$ we have $\Gamma_{d,(1,\dots,d)} (\sqrt{A_n}) -
b_n^{-2}4(A_n)_{K,d}(A_n)_{d,K}\in\mathcal{M}_{(d-1)}$. Since also
$2(A_n)_{K,d} + x_K - s\mathbf{1} + b_n^{-2}s2(A_n)_{K,d}$ converges
to $2(A_0)_{K,d} + x_K - s\mathbf{1}$ as $n\to\infty$, we conclude
that the survivor functions converge and consequently $p_n(A_n) \to
p(A_0)$. Applying Theorem~5.5 in Billingsley \cite{bil1968}
yields \eqref{int_conv}.

If not all random vectors $\mathbf{X}_{i,n}$ are non-degenerate, then
it follows from the weak convergence $\eta^2_n\Rightarrow\eta^2$
that $\eta_n^2(D\setminus D_0) \to\eta^2(D\setminus D_0) = 0$, as
$n\to\infty$. Indeed, since $D\setminus D_0$ is closed in $D$, we
have that $\eta^2(\partial(D\setminus D_0)) = 0$. Thus, the
degenerate random vectors in \eqref{suff_conv} are negligible. This
concludes the proof.
\end{pf*}
\begin{pf*}{Proof of Proposition~\ref{prop1}}
Let $t_1,\dots,t_m\in\R^d$ and $x_1,\dots,x_m\in\R$ be fixed. It
follows from
formula (19) in Kabluchko \cite{kab2011} that for a fixed
variogram $\gamma_0\in V_d$,
the finite dimensional distribution $(\xi(t_1),\dots,\xi(t_m))$ of the
corresponding Brown--Resnick process in \eqref{BRproc} is given by
$H_{\Lambda_{\gamma_0}}$ with
$\Lambda_{\gamma_0} = (\sqrt{\gamma_0(t_j-t_k)}/4)_{1\leq j,k \leq
m}$.

For the max-mixture of Brown--Resnick processes w.r.t. the mixture measure
$\mathbb{Q}$, we obtain via void probabilities of Poisson point processes
\begin{eqnarray*}
-\log\P \bigl(\xi_{\mathbb{Q}}(t_1)\leq x_1, \dots,
\xi_{\mathbb
{Q}}(t_m)\leq x_m \bigr) &=& \int
_{\mathbb{R}} \mathrm{e}^{-u} \P \Bigl( u > \min
_{i=1,\dots,m} x_i - W_\gamma(t_i) +
2\gamma(t_i) \Bigr)
\\
&=& \mathbb E \max_{i=1,\dots,m} \exp \bigl(W_\gamma(t_i)
- 2\gamma (t_i) - x_i \bigr),
\end{eqnarray*}
where $\gamma$ has distribution $\mathbb Q$ and the process $W_\gamma
$, conditional on $\gamma$,
is a zero-mean Gaussian process with stationary increments, variogram
$4\gamma$
and $W_\gamma(0) = 0$ a.s. By conditioning on the variogram we get
%
\begin{eqnarray}\label{mixing}
&&-\log\P \bigl(\xi_{\mathbb{Q}}(t_1)\leq
x_1, \dots,\xi _{\mathbb{Q}}(t_m)\leq
x_m \bigr)
\nonumber\\
&&\quad  = \int_{V_d} \mathbb E \max_{i=1,\dots,m}
\exp \bigl(W_{\gamma_0}(t_i) - 2\gamma_0(t_i)
- x_i \bigr) \mathbb{Q}(\mathrm{d}\gamma _0)
\\
&&\quad  = \int_{V_d} \mathbb-\log
H_{\Lambda_{\gamma
_0}}(x_1,\dots,x_m) \mathbb{Q}(\mathrm{d}
\gamma_0).\nonumber
\end{eqnarray}
Thus, comparing with \eqref{fdd}, the finite dimensional distributions
of $\xi_{\mathbb{Q}}$ are given
by the max-mixtures of Brown--Resnick processes w.r.t. the mixture
measure $\mathbb Q$.
Consequently, $\xi_{\mathbb{Q}}$ is max-stable and by \eqref{mixing},
stationarity is preserved under max-mixing.
\end{pf*}

\section*{Acknowledgements}
The authors are grateful to Kirstin Strokorb for helpful discussions on
extremal correlation functions and
to two unknown referees who helped to considerably improve the paper.
S. Engelke has been financially supported by
Deutsche Telekom Stiftung and the Swiss National Science Foundation project
200021-134785. M. Schlather
has been financially supported by Volkswagen Stiftung within the
project `WEX-MOP'.




\printhistory

\end{document}